\documentclass[a4paper,12pt]{article}

\usepackage[utf8]{inputenc}
\usepackage[T1]{fontenc}
\usepackage{amssymb,amsmath,amsthm,algorithmic,url,tikz,graphics}
\usetikzlibrary{decorations.pathreplacing}
\newtheorem{algorithm}{Algorithm}
\newtheorem{theorem}{Theorem}
\newtheorem{lemma}{Lemma}
\newtheorem{remark}{Remark}
\newtheorem{property}{Property}
\usepackage{fullpage}
\usepackage[pdftex]{hyperref}

\DeclareMathOperator{\RN}{RN}
\DeclareMathOperator{\fl}{float}

\title{\bf On the maximum relative error when computing $x^n$ in floating-point arithmetic}

\author{Stef Graillat\\Université Pierre et Marie Curie Paris 6 \\ Laboratoire LIP6\\ \and Vincent Lefèvre\\ Inria, Laboratoire LIP\\Université de Lyon \and Jean-Michel Muller\\ CNRS, Laboratoire LIP\\Université de Lyon}

\begin{document}

\maketitle

\begin{abstract}
In this paper, we improve the usual relative error bound for the computation of $x^n$
through iterated multiplications by $x$ in binary floating-point
arithmetic. The obtained error bound is only slightly better than the
usual one, but it is simpler. We also discuss the more general problem
of computing the product of $n$ terms. 
\end{abstract}

\textbf{Keywords:}  floating-point arithmetic, rounding error,
accurate error bound, exponentiation \\

\textbf{AMS Subject Classifications:} 15-04, 65G99, 65-04 \\

\section{Introduction}

\subsection{Floating-point arithmetic and rounding errors}

In general, computations in floating-point arithmetic are not errorless: a small rounding error occurs  each time an arithmetic operation is performed. Depending on the calculation being done, the global influence of these individual rounding errors can rank anywhere between completely negligible and overwhelming. Hence, it is always important to have some information on the numerical quality of a computed result. Furthermore, when critical applications are at stake, one may need \emph{certain} yet \emph{tight} error bounds. The manipulation of these error bounds (either paper-and-pencil manipulation or---if one wishes to do some dynamical error analysis---numerical manipulation) will also be made easier if these bound are \emph{simple}.

In the following, we assume a radix-$2$, precision-$p$, floating-point (FP) arithmetic. To simplify the presentation, we assume an unbounded exponent range: our results will be applicable to ``real life'' floating-point systems, such as those that are compliant with the IEEE 754-2008 Standard for Floating-Point Arithmetic~\cite{IEEE754-2008,MullerEtAl2010}, provided that no underflow or overflow occurs. In such an arithmetic, a floating-point number is either zero or a number of the form
\[
x = X \cdot 2^{e_x-p+1},
\]
where $X$ and $e_x$ are integers, with $2^{p-1} \leq  |X| \leq 2^p-1$. The number $X$ is called the \emph{integral significand} of $x$, $X \cdot{} 2^{-p+1}$ is called the \emph{significand} of $x$, and $e_x$ is called the \emph{exponent} of $x$.

As said above, since in general the sum, product, quotient, etc., of two FP numbers is not a FP number, it must be \emph{rounded}. The IEEE 754-2008 Standard requires that the arithmetic operations should be \emph{correctly rounded}: a rounding function must be chosen among five possible functions defined by the standard. If $\circ{}$ is the rounding function, when the arithmetic operation $(a \top b)$ is performed, the value that must be returned is the FP number $\circ{}(a \top b)$. The default rounding function is \emph{round to nearest ties to even}, denoted $\RN_{even}$, defined as follows: 
\begin{itemize}
     \item[($i$)] for all FP numbers $y$, $|\RN_{even}(t) -t| \leq |y-t|$;
     \item[($ii$)] if there are two FP numbers that satisfy ($i$), $\RN_{even}(t)$ is the one whose integral significand is even.
\end{itemize}
The IEEE 754-2008 standard defines another round-to-nearest rounding function, namely \emph{round to nearest ties to away}, where ($ii$) is replaced by
\begin{itemize}
     \item[($ii'$)] if there are two FP numbers that satisfy ($i$), $\RN_{away}(t)$ is the one whose integral significand has the largest magnitude.
\end{itemize}
In the following, $\RN$ is one of these two round-to-nearest functions. More precisely: unless stated otherwise, the bounds we give are applicable to both rounding functions. However, when we build examples (for instance for checking how tight are the obtained bounds), we use $\RN_{even}$.

%
%
%

Recently, classic error bounds for summation and dot product have been
improved by  Jeannerod and Rump~\cite{Rump2012,JR2013}. They  have considered the problem of calculating the sum of $n$ FP numbers $x_1, x_2, \ldots{}, x_n$. If we call $\fl( \sum_{i=1}^n x_i )$ the computed result and $u = 2^{-p}$ the \emph{rounding unit}, they have 
shown that 
\begin{equation}
\left| \fl\left( \sum_{i=1}^n x_i \right) - \sum_{i=1}^n x_i \right| \leq (n-1)
\cdot u \sum_{i=1}^n |x_i|
\end{equation}
which is better than the previous bound~\cite[p.63]{Hig02}
\[
\left| \fl\left( \sum_{i=1}^n x_i \right) - \sum_{i=1}^n x_i \right| \leq \gamma_{n-1} \sum_{i=1}^n |x_i|
\]
where 
\begin{equation}
\label{defgamma}
\gamma_n = \frac{n \cdot u}{1 - n \cdot u} = n \cdot u + n^2 \cdot u^2 + n^3 \cdot u^3 + \cdots{} = n \cdot u + \mathcal{O}(u^2).
\end{equation}

We are interested in finding if a similar simplification is possible in  the particular case of the computation of an integer power $x^n$, that is we wish to know if the result computed using the ``naive algorithm'' (Algorithm~\ref{alg:naive-power} below) is always within relative error $(n-1) \cdot{}u$ from the exact result.
This is ``experimentally true'' in binary32/single precision arithmetic. More precisely, we did an exhaustive check for all $x \in [1;2[$ in
binary32 ($2^{23}$ numbers to be checked) until overflow for $x^n$.
For the smallest number larger than $1$, namely $x = 1+ 2 u$, $n \approx 7.5 \times 10^8$ is needed to reach overflow.
Our test used a $100$-bit interval arithmetic  provided by the
MPFI~\cite{MPFI} package. 

In this paper, we  prove---under mild hypotheses---that this result holds for all ``reasonable'' floating-point formats (we need the precision $p$ to be larger than or equal to $5$, which is always true in practice).

%

%
%

\subsection{Relative error due to roundings}
\label{sec:relative-error-classical-general}

Let $t$ be a positive real number between $2^e$ and $2^{e+1}$, where
$e \in \mathbb{Z}$. The rounding $\RN(t)$ is between $2^e$ and $2^{e+1}$ too, and we have
\begin{equation}
\label{eq:simplest-rounding-error}
|\RN(t) - t| \leq 2^{e-p}.
\end{equation}
From this, we easily deduce a bound on the relative error due to rounding $t$
\begin{equation}
\label{eq:relative-error-bound-rounding-general}
\left|\frac{\RN(t)-t}{t}\right| \leq 2^{-p} = u.
\end{equation}

This is illustrated by Figure~\ref{fig:illustration-u}.

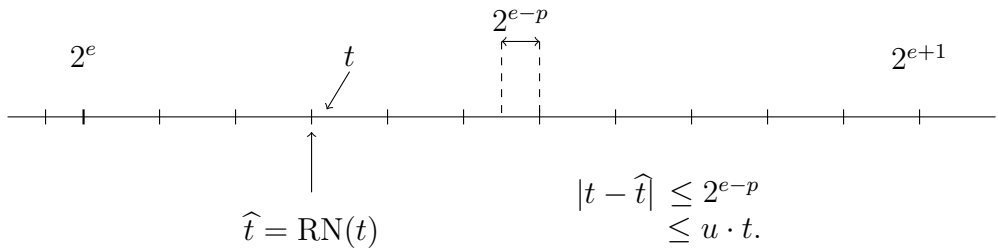
\begin{figure}[htb]
\begin{tikzpicture}
\draw(0,0)--(13,0);
\draw(0.5,-0.1)--(0.5,+0.1);
\draw[thick](1,-0.1)--(1,+0.1);
\foreach\i in {2,3,...,12}{\draw(\i,-0.1)--(\i,+0.1);}
\draw[<->](6.5,1)--(7,1);
\draw[dashed](6.5,1)--(6.5,0);
\draw[dashed](7,1)--(7,0);
\node(halfulp)at(6.75,1.3){$2^{e-p}$};
\node(2e)at(1,0.8){$2^e$};
\node(2ep1)at(12,0.8){$2^{e+1}$};
\node(hatt)at(4,-1.5){$\widehat{t} = \RN(t)$};
\node(t)at(4.5,0.8){$t$};
\draw[->](4.5,0.6)--(4.2,0.1);
\draw[->](4,-1)--(4,-0.2);
\node(equation1)at(8,-1){$|t-\widehat{t}|$};
\node(equation2)at(9.3,-1){$ \leq 2^{e-p}$};
\node(equation3)at(9.3,-1.5){$\leq u \cdot t.$};
\end{tikzpicture}
\caption{\emph{In precision-$p$ binary floating-point arithmetic, in the normal range, the relative error due to rounding to nearest is always bounded by $u = 2^{-p}$.}}
\label{fig:illustration-u}
\end{figure}

For instance, when we perform a floating-point multiplication, if $a$ and $b$ are the input FP operands, $z = ab$ is the exact result, and $\widehat{z} = \RN(z)$ is the computed result, then we have

\begin{equation}
\label{error-one-multiplication}
(1-u) \cdot{} z \leq \widehat{z} \leq (1+u) \cdot z.
\end{equation}

Assume that we wish to evaluate the product
\[
a_1 \cdot{} a_2 \cdots{}  a_n,
\]
of $n$ floating-point numbers, and that the product is evaluated as
\begin{equation}\label{eq:prod}
\RN( \cdots{} \RN(\RN(a_1 \cdot{} a_2) \cdot{} a_3) \cdot{} \cdots{} ) \cdot{} a_n).
\end{equation}
Define $\pi_n$ as the exact value of $a_1 \cdots a_n$, and $\widehat{\pi}_n$ as the computed value. A simple induction, based on (\ref{error-one-multiplication}), allows one to show
\begin{theorem}
\label{thm:classical-theorem}
Let $a_1,\ldots, a_n$ be floating-point numbers, $\pi_n = a_1 \cdots
a_n$, and $\widehat{\pi}_n$ the computed value using
\eqref{eq:prod}. Then we have
\begin{equation}
\label{eq:classical-thm}
(1 - u)^{n-1} \pi_n \leq \widehat{\pi}_n \leq (1+u)^{n-1} \pi_n.
\end{equation}
\end{theorem}
See~\cite{GR09} for some results concerning the computation of the
product of floating-point numbers.
Therefore, the relative error of the computation, namely $|\widehat{\pi_n}-\pi_n|/\pi_n$ is upper-bounded by
\[
\psi_{n-1} = (1+u)^{n-1}-1.
\]
One easily shows that, as long as $ku < 1$ (which always holds in practical cases),
\[
k \cdot u \leq \psi_k \leq \gamma_k,
\]
where $\gamma_k$ is defined by \eqref{defgamma}.
Although the bound $\psi_{n-1}$ on the relative error of the computation of $a_1 \cdot{} a_2 \cdots a_n$ is very slightly\footnote{As long as $nu$ is small enough in front of $1$.}  better than $\gamma_{n-1}$, the classical bound found in the literature is $\gamma_{n-1}$. The reason for this is that it is easier to manipulate in calculations.

And yet, in all our experiments, we observed a relative error less than $(n-1) \cdot u$. If we could prove that this is a valid bound, this would be even easier to manipulate. In the general case of an iterated product, we did not succeed in proving that. We could only automatically build cases, for each value of the precision $p$, for which the attained relative error is extremely close to, yet not larger than, $(n-1) \cdot u$ (see Section~\ref{sec:iterated-products}). However, in the particular case $n \leq 4$, one can prove that the relative error is less than $(n-1) \cdot u$. This is done as follows.

First, as noticed by Jeannerod and Rump~\cite{JeannerodRump2014}, one may remark that the bound on the relative error due to rounding---i.e., (\ref{eq:relative-error-bound-rounding-general})---can be slightly improved. Assume that $t$ is a real number between $2^e$ and $2^{e+1}$. We already know that $|t-\RN(t)| \leq 2^{e-p} = u \cdot{} 2^e$. Therefore:
\begin{itemize}
    \item if $t \geq 2^e \cdot{} (1+u)$, then $|t-\RN(t)|/t \leq u/(1+u)$;
    \item if $t <  2^e \cdot{} (1+u)$, then $\RN(t) = 2^e$. Let $t = 2^e \cdot (1 + \tau \cdot u)$, we have:
    $|t-\RN(t)|/t = \tau \cdot u / (1 + \tau \cdot u)$. An elementary study shows that for $\tau \in [0,1)$, $\tau \cdot u / (1 + \tau \cdot u) < u / (1+u)$.
\end{itemize}
Therefore the maximum relative error due to rounding is bounded\footnote{Incidentally, if $\RN = \RN_{even}$, that error is attained when $t = 1+u$, which shows that the bound cannot be improved further. } by $u/(1+u)$.
A consequence of this is that  $u$ can be replaced by $u/(1+u)$ in
(\ref{eq:classical-thm}). This is illustrated by Figure~
\ref{fig:wobbling-relative-error} (see p.~\pageref{fig:wobbling-relative-error}). In the general case (that is, for any $n$), this improvement does not suffice to show Theorem~\ref{thm:main-theorem}, and yet, when $n \leq 4$, we can use the following result.

\begin{property}
\label{property-use-of-u-over-1-plus-u}
If $k \leq 3$ then
\[
\left(1 + \frac{u}{1+u}\right)^k < 1 + k \cdot{} u.
\]
\end{property}
\begin{proof}
The  simplest way to prove Property~\ref{property-use-of-u-over-1-plus-u} is to separately consider the cases $k = 1, 2,$ and $3$:
\begin{itemize}
    \item the case $k=1$ is straightforward:
     \item if $k=2$, we have
         \[
\left(1 + \frac{u}{1+u}\right)^2 - (1 + 2u) = - \frac{u^2 \cdot{} (1+2u)}{(1+u)^2} < 0;
\]
     \item if $k=3$, we have
\[
\left(1 + \frac{u}{1+u}\right)^3 - (1 + 3u) = -\frac{u^3 \cdot (3u+2)}{(1+u)^3} < 0.
\]
\end{itemize}
\end{proof}
By taking $k = n-1$, we immediately deduce that for $n \leq 4$, the relative error of the iterative product of $n$ FP numbers is bounded by $(n-1) \cdot u$.

Although we conjecture that this remains true for larger values of $n$, we did not succeed in proving that (notice that Property~\ref{property-use-of-u-over-1-plus-u} is no longer true when $k \geq 4$).
However, in the particular case of the computation of $x^n$, for some given FP number $x$ and some positive integer $n$, we could prove the bound $(n-1) \cdot u$: our main result is Theorem~\ref{thm:main-theorem} below.

\subsection{The particular case of computing powers}

In the following, we are interested in computing $x^n$, where $x$ is a FP number and $n$ is an integer. It is not difficult to show by induction that the bound provided by Theorem~\ref{thm:classical-theorem} applies not only to the case that was discussed above (computation of $\RN( \cdots{} \RN(\RN(x \cdot{} x) \cdot{} x) \cdot{} \cdots{} ) \cdot{}x$) but to the larger class of recursive algorithms where the approximation to $x^{k+\ell}$ is deduced from  approximations to $x^k$ and $x^{\ell}$ by a FP multiplication. However, we will prove a (slightly) better bound only in the case where the algorithm used for computing $x^n$ is Algorithm~\ref{alg:naive-power} below.

\begin{samepage}
\begin{algorithm}[naive-power$(x,n)$]~\\[-0.5cm]
\label{alg:naive-power}
  \begin{algorithmic}[0] 
\STATE{$y \gets x$}
\FOR{$k = 2$ to $n$}
  \STATE{$y \gets \RN(x \cdot y)$}
\ENDFOR
\RETURN{y}
\end{algorithmic}
\end{algorithm}
\end{samepage}
We will define $\widehat{x}_j$ as the value of variable $y$ after the iteration corresponding to $k=j$ in the \textbf{for} loop of Algorithm~\ref{alg:naive-power}. We have $\widehat{x}_2 = \RN(x^2)$, and $\widehat{x}_k = \RN(x \cdot \widehat{x}_{k-1})$.
We wish to prove

\begin{theorem}
\label{thm:main-theorem}
Assume $p \geq 5$ (which holds in all practical cases). If \[n \leq \sqrt{2^{1/2}-1} \cdot 2^{p/2},\] then
\[
\left| \widehat{x}_n - x^n \right| \leq (n-1) \cdot u \cdot x^n.
\]
\end{theorem}

To prove Theorem~\ref{thm:main-theorem}, it suffices to prove it in the case $1 \leq x < 2$: in the following we will therefore assume that $x$ lies in that range.

We prove Theorem~\ref{thm:main-theorem} in Section~\ref{sec:proof}. Before that, in Section \ref{sec:prem}, we give some preliminary results. In Section \ref{sec:thm}, we discuss
the tightness of our new bound.  Section \ref{sec:prod} is devoted to a
discussion on the possible generalization of this bound to the product
of $n$ floating-point numbers.

\section{Preliminary results} \label{sec:prem}

In this section we give some preliminary results that will help to
improve the bound of Theorem~\ref{thm:classical-theorem} in the
specific case of the computation of integer powers. Let us start with an easy
remark. 

\begin{remark}
Since $(1-u)^{n-1} \geq 1 - (n-1) \cdot u$ for all $n \geq 2$ and $u \in [0,1]$, the left-hand bound of (\ref{eq:classical-thm}) suffices to show that
    $(1 - (n-1) \cdot u) \cdot x^n \leq \widehat{x}_n$. In other words, to establish Theorem~\ref{thm:main-theorem}, we only need to improve on the right-hand bound of (\ref{eq:classical-thm}).
\end{remark}

Now, for $t \neq 0$, define
\[
\overline{t} = \frac{t}{2^{\lfloor \log_2 |t| \rfloor}}.
\]

We have,
\begin{lemma}
\label{lemma:significand-above-w}
Let $t$ be a real number. If
\begin{equation}
\label{eq:signif-bounded-by-w}
2^e \leq w \cdot 2^e \leq |t| < 2^{e+1}, e \in \mathbb{Z}
\end{equation}
(in other words, if $|\overline{t}|$ is lower-bounded by w)
then
\[
\left| \frac{\RN(t) - t}{t} \right| \leq \frac{u}{w}.
\]
\end{lemma}

Figure~\ref{fig:wobbling-relative-error} illustrates Lemma~\ref{lemma:significand-above-w}, and Figure~\ref{fig:wobbling-2} illustrates this ``wobbling'' maximal relative error due to rounding.

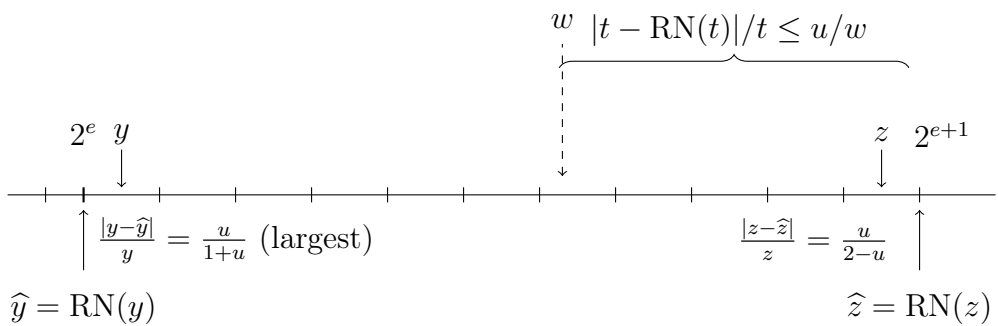
\begin{figure}[htb]
\begin{tikzpicture}
\draw(0,0)--(13,0);
\draw(0.5,-0.1)--(0.5,+0.1);
\draw[thick](1,-0.1)--(1,+0.1);
\foreach\i in {2,3,...,12}{\draw(\i,-0.1)--(\i,+0.1);}
\node(2e)at(1,0.8){$2^e$};
\node(2ep1)at(12.3,0.8){$2^{e+1}$};
\node(hatt)at(1,-1.5){$\widehat{y} = \RN(y)$};
\node(y)at(1.5,0.8){$y$};
\draw[->](1.5,0.6)--(1.5,0.1);
\draw[->](1,-1)--(1,-0.2);
\draw[dashed,<-](7.3,0.25)--(7.3,2.0);
\node(w)at(7.3,2.3){$w$};
\draw[decorate,decoration={brace,amplitude=5pt},xshift=-4pt,yshift=-9pt]
(7.4,2) -- (12,2) node [black,midway,above=4pt,xshift=-2pt] {$|t-\RN(t)|/t \leq u/w$};
\node(equation1)at(3,-0.6){$\frac{|y-\widehat{y}|}{y} = \frac{u}{1+u}$ (largest)};
\node(hatz)at(12,-1.5){$\widehat{z} = \RN(z)$};
\node(t)at(11.5,0.8){$z$};
\draw[->](11.5,0.6)--(11.5,0.1);
\draw[->](12,-1)--(12,-0.2);
\node(equation2)at(10.6,-0.6){$\frac{|z-\widehat{z}|}{z} = \frac{u}{2-u}$};
\end{tikzpicture}
\caption{\emph{The bound on the relative error due to rounding to nearest can be reduced to $u/(1+u)$. Furthermore, if we know that $\overline{t} = t/2^e$ is larger than $w$, then $|\RN(t)-t|/t$ is less than $u/w$.}}
\label{fig:wobbling-relative-error} 
\end{figure}

\begin{figure}[htb]
\includegraphics[scale=0.75]{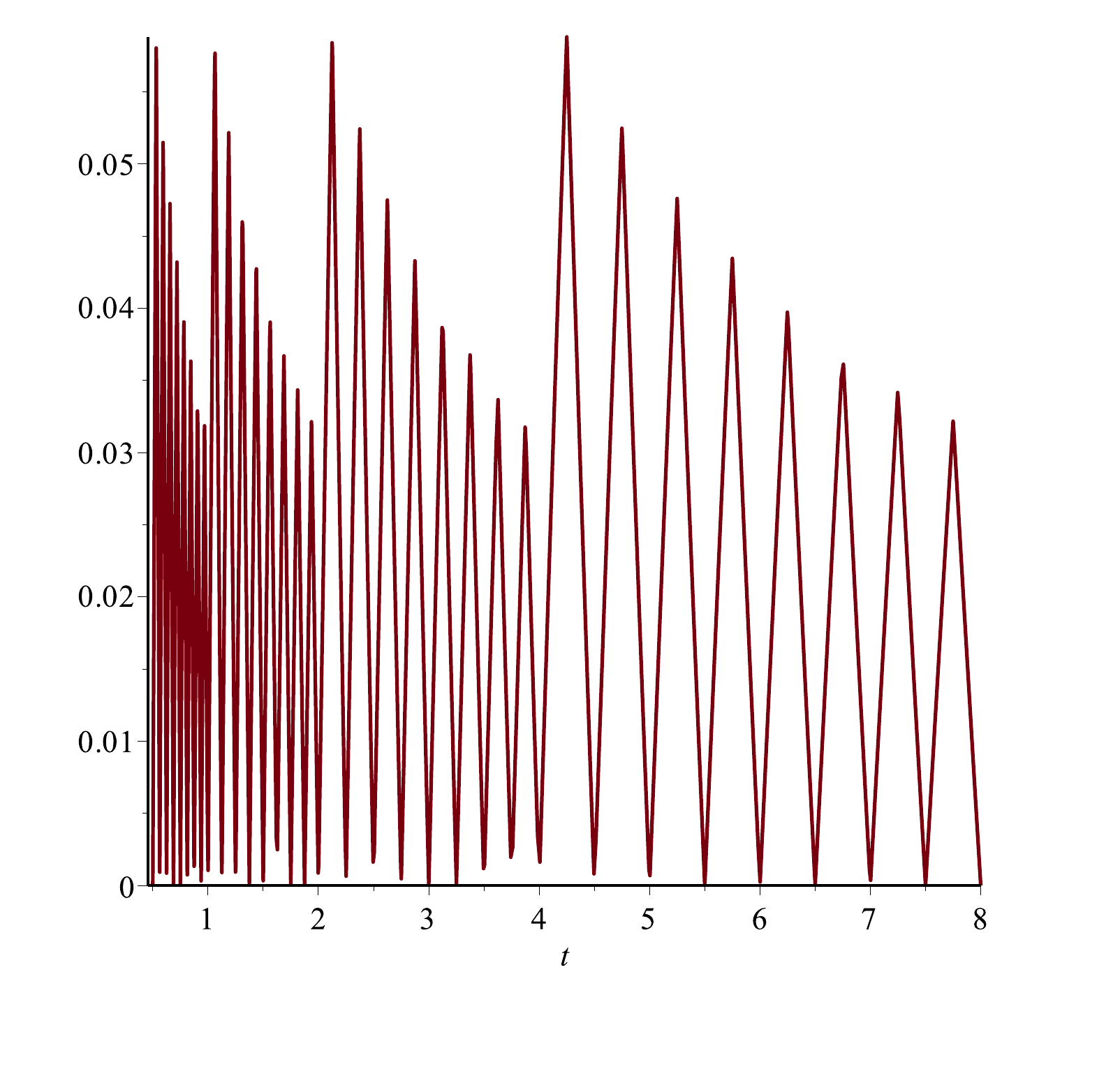}
\caption{The relative error due to rounding, namely $|\RN(t)-t|/t$, for $t$ between $1/5$ and $8$.}
\label{fig:wobbling-2} 
\end{figure}


%

Lemma~\ref{lemma:significand-above-w} is an immediate consequence of (\ref{eq:simplest-rounding-error}) and (\ref{eq:signif-bounded-by-w}). It is at the heart of our study: our problem will be to show that at least once in the execution of Algorithm~\ref{alg:naive-power} the number $x \cdot y$ is such that $\overline{x \cdot y}$ is large enough to sufficiently reduce the error bound on the corresponding FP multiplication $y \gets \RN(x \cdot y)$, so that the overall relative error bound becomes smaller than $(n-1) \cdot{} u$. More precisely, we will show that, under some conditions, at least once, $\overline{x \cdot y}$ is larger than $1+n^2u$, so that in (\ref{eq:classical-thm}) the term $(1+u)^{n-1}$ can be replaced by
\[
(1+u)^{n-2} \cdot \left(1 + \frac{u}{1+n^2u}\right).
\]
Therefore, we need to bound this last quantity. We have,

\begin{lemma}
\label{lemma:n:u}
If $0 \leq u \leq 2/(3 n^2)$ then
\begin{equation}
\label{eq:lemma:n:u}
(1+u)^{n-2} \cdot \left(1 + \frac{u}{1+n^2u}\right) \leq 1 + (n-1) \cdot u.
\end{equation}
\end{lemma}

\begin{proof}
Proving Lemma~\ref{lemma:n:u} reduces to proving that the polynomial
\[
P(u) = (1+(n-1)u)(1+n^2u)-(1+u)^{n-2}(1+n^2u+u)
\]
is $\geq 0$ for $0 \leq u \leq 2/(3n^2)$.

Notice that for $u \geq 0$, we have
\[
\ln(1+u) \leq u - \frac{u^2}{2} + \frac{u^3}{3}.
\]
From $\ln(1+u) \leq u$ we also deduce that $(n-2) \ln(1+u) \leq (n-2) u \leq 1/(2n)$.
For $0 \leq t \leq 1/6$, $e^t \leq 1 + t + \frac{3}{5} t^2$.
Therefore, for $0 \leq u \leq 2/3n^2$, to prove that $P(u) \geq 0$ it suffices to prove that
\begin{equation}
\label{reduc-1}
\begin{array}{c}
Q(n,u) = \left( 1+ \left( n-1 \right) u \right)  \left( {n}^{2}u+1 \right)\\
 -
 \left( 1+ \left( n-2 \right)  \left( u-1/2\,{u}^{2}+1/3\,{u}^{3}
 \right) +3/5\, \left( n-2 \right) ^{2} \left( u-1/2\,{u}^{2}+1/3\,{u}
^{3} \right) ^{2} \right)  \\
\times
\left( {n}^{2}u+u+1 \right) \geq 0.
\end{array}
\end{equation}
By defining $a = n^2u$, $Q(n,u) = R(n,a)$, with
\renewcommand{\arraystretch}{1.8}
\begin{equation}
\label{eq:def-R}
\begin{array}{c}
R(n,a) = -\frac{1}{5}\,{\frac {{a}^{2} \left( 3\,a-2 \right) }{{n}^{2}}}+\frac{1}{10}\,{\frac {
{a}^{2} \left( 29\,a+19 \right) }{{n}^{3}}}+\frac{1}{5}\,{\frac {{a}^{2}
 \left( 3\,{a}^{2}-17\,a-7 \right) }{{n}^{4}}} \\
 -\frac{1}{30}\,{\frac {{a}^{3}
 \left( 82\,a-5 \right) }{{n}^{5}}}-{\frac {1}{60}}\,{\frac {{a}^{3}
 \left( 33\,{a}^{2}-187\,a+20 \right) }{{n}^{6}}}+\frac{1}{15}\,{\frac {{a}^{4
} \left( 33\,a-8 \right) }{{n}^{7}}}\\
+{\frac {1}{60}}\,{\frac {{a}^{4}
 \left( 12\,{a}^{2}-153\,a+52 \right) }{{n}^{8}}}-\frac{1}{5}\,{\frac {{a}^{5}
 \left( 4\,a-7 \right) }{{n}^{9}}}-\frac{1}{15}\,{\frac {{a}^{5} \left( {a}^{2
}-14\,a+21 \right) }{{n}^{10}}}\\
+{\frac {4}{15}}\,{\frac {{a}^{6}
 \left( a-2 \right) }{{n}^{11}}}-\frac{1}{15}\,{\frac {{a}^{6} \left( 5\,a-8
 \right) }{{n}^{12}}}\\
 +{\frac {4}{15}}\,{\frac {{a}^{7}}{{n}^{13}}}-{
\frac {4}{15}}\,{\frac {{a}^{7}}{{n}^{14}}}
\end{array}
\end{equation}
Multiplying $R(n,a)$ by $5n^2/a^2$, we finally obtain
\begin{equation}
\label{eq:def-S}
\begin{array}{c}
S(n,a) = -3\,a+2+ \left( {\frac {29}{2}}\,a+\frac{19}{2} \right) {n}^{-1}+{\frac {3\,{a
}^{2}-17\,a-7}{{n}^{2}}}-\frac{1}{6}\,{\frac {a \left( 82\,a-5 \right) }{{n}^{
3}}} \\
-\frac{1}{12}\,{\frac {a \left( 33\,{a}^{2}-187\,a+20 \right) }{{n}^{4}}}+
\frac{1}{3}\,{\frac {{a}^{2} \left( 33\,a-8 \right) }{{n}^{5}}}+\frac{1}{12}\,{\frac {
{a}^{2} \left( 12\,{a}^{2}-153\,a+52 \right) }{{n}^{6}}} \\
-{\frac {{a}^{
3} \left( 4\,a-7 \right) }{{n}^{7}}}-\frac{1}{3}\,{\frac {{a}^{3} \left( {a}^{
2}-14\,a+21 \right) }{{n}^{8}}}+\frac{4}{3}\,{\frac {{a}^{4} \left( a-2
 \right) }{{n}^{9}}}-\frac{1}{3}\,{\frac {{a}^{4} \left( 5\,a-8 \right) }{{n}^
{10}}}\\+\frac{4}{3}\,{\frac {{a}^{5}}{{n}^{11}}}-\frac{4}{3}\,{\frac {{a}^{5}}{{n}^{12}
}}
\end{array}
\end{equation}
\renewcommand{\arraystretch}{1.0}
We wish to show that $S(n,a) \geq 0$ for $0 \leq a \leq 2/3$. Let us examine the terms of $S(n,a)$ separately. For $a$ in the interval $[0,2/3]$ and $n \geq 3$:
\begin{itemize}
    \item the term $-3\,a+2$ is always larger than $0$;
    \item the term $\left( {\frac {29}{2}}\,a+\frac{19}{2} \right) {n}^{-1}$ is always larger than $19/(2n)$;
    \item the term ${\frac {3\,{a}^{2}-17\,a-7}{{n}^{2}}}$ is always larger than $-6/n$;
    \item the term $-\frac{1}{6}\,{\frac {a \left( 82\,a-5 \right) }{{n}^{3}}}$ is always larger than $-7/(10n)$;
    \item the term $-\frac{1}{12}\,{\frac {a \left( 33\,{a}^{2}-187\,a+20 \right) }{{n}^{4}}}$ is always larger than $-17/(10000n)$;
    \item the term $\frac{1}{3}\,{\frac {{a}^{2} \left( 33\,a-8 \right) }{{n}^{5}}}$ is always larger than $-3/(10000n)$;
    \item the term $\frac{1}{12}\,{\frac {{a}^{2} \left( 12\,{a}^{2}-153\,a+52 \right) }{{n}^{6}}}$ is always larger than $-69/(10000n)$;
    \item the term $-{\frac {{a}^{3} \left( 4\,a-7 \right) }{{n}^{7}}}$ is always larger than $0$;
    \item the term $-\frac{1}{3}\,{\frac {{a}^{3} \left( {a}^{2}-14\,a+21 \right) }{{n}^{8}}}$ is always larger than $-6/(10000n)$;
    \item the term $\frac{4}{3}\,{\frac {{a}^{4} \left( a-2 \right) }{{n}^{9}}}$ is always larger than $-6/(100000n)$;
    \item  the term $-\frac{1}{3}\,{\frac {{a}^{4} \left( 5\,a-8 \right) }{{n}^{10}}}$ is always larger than $0$;
    \item the term $\frac{4}{3}\,{\frac {{a}^{5}}{{n}^{11}}}$  is always larger than $0$;
    \item the term $-\frac{4}{3}\,{\frac {{a}^{5}}{{n}^{12}}}$ is always larger than $-1/(1000000n)$.
\end{itemize}
By summing all these lower bounds, we find that for $0 \leq a \leq 2/3$ and $n \geq 3$, $S(n,a)$ is always larger than $2790439/(1000000n)$. 
\end{proof}

Let us now raise some remarks, that are direct consequences of Lemma~\ref{lemma:n:u}.

\begin{remark}
\label{rem:rounding-downwards}
Assume $n \leq \sqrt{2/3} \cdot{} 2^{p/2}$. If for some $k \leq n$, we have $\RN(x \cdot \widehat{x}_{k-1}) \leq x \cdot \widehat{x}_{k-1}$ (i.e., if in Algorithm~\ref{alg:naive-power} at least one rounding is done downwards), then $\widehat{x}_n \leq (1+(n-1) \cdot u) x^n$.
\end{remark}
\begin{proof}
We have
    \[
    \widehat{x}_n \leq (1+u)^{n-2} x^n.
    \]
    Lemma \ref{lemma:n:u} implies that $(1+u)^{n-2}$ is less than $1+(n-1) \cdot u$. Therefore,   
    \[
    \widehat{x}_n \leq (1+(n-1) \cdot u) x^n.
    \]
\end{proof}

\begin{remark}
\label{rem:significand-large-enough}
Assume $n \leq \sqrt{2/3} \cdot{} 2^{p/2}$. If there exists $k$, $1 \leq k \leq n-1$, such that 
    $
    \overline{x \cdot \widehat{x}_k} \geq 1 + n^2 \cdot u,
    $
    then $\widehat{x}_n \leq (1+(n-1) \cdot u) x^n.$
\end{remark}
\begin{proof}
By combining Lemma~\ref{lemma:significand-above-w} and Lemma~\ref{lemma:n:u}, if there exists $k$, $1 \leq k \leq n-1$, such that 
    \[
    \overline{x \cdot \widehat{x}_k} \geq 1 + n^2 \cdot u,
    \]
    then 
    \[
    \widehat{x}_n \leq (1+u)^{n-2} \cdot \left(1 + \frac{u}{1+n^2u}\right)  \cdot x^n \leq (1 + (n-1) \cdot u) \cdot x^n.
    \]
\end{proof}

\section{Proof of Theorem~\ref{thm:main-theorem}} \label{sec:proof}

The proof is articulated as follows
\begin{itemize}
    \item first, we show that if $x$ is close enough to $1$, then when
      computing $\RN(x^2)$, the rounding is done downwards (i.e.,
      $\RN(x^2) \leq x^2$), which implies, from
      Remark~\ref{rem:rounding-downwards}, that $\widehat{x}_n \leq
      (1+(n-1) \cdot u) x^n$. This is the purpose of
      Lemma~\ref{lemma:smallx}. 
    \item then, we show that in the other cases, there is at least one
      $k \leq n-1$ such that $\overline{x \cdot \widehat{x}_k} \geq 1
      + n^2 \cdot u$, which implies, from
      Remark~\ref{rem:significand-large-enough}, that $\widehat{x}_n
      \leq (1+(n-1) \cdot u) x^n$. 
\end{itemize}

\begin{lemma}\label{lemma:smallx}
Let $x = 1 + k \cdot 2^{-p+1} = 1 + 2ku, k \in \mathbb{N}$ (all FP numbers between $1$ and $2$ are of that form). We have $x^2 = 1 + 2k \cdot{} 2^{-p+1} + k^2 \cdot 2^{-2p+2}$, so that if $k < 2^{p/2-1}$, i.e., if $ 1 \leq x < 1 + 2^{p/2}u$, then $\widehat{x}_2 = 1 +2 k \cdot 2^{-p+1} < x^2$, which, by  Remark~\ref{rem:rounding-downwards}, implies $\widehat{x}_n \leq (1 + (n-1) u) \cdot x^n$.
\end{lemma}

%

Remark~\ref{rem:significand-large-enough} and
Lemma~\ref{lemma:smallx} imply that to prove
Theorem~\ref{thm:main-theorem}, we are reduced to examine the case
where $1 + 2^{p/2}u \leq x <2$
and we assume $u \leq 2/(3n^2)$, i.e., $n < \sqrt{2/3} \cdot{}
2^{p/2}$ (later on, we will see that a stronger assumption is
necessary). For that, we distinguish between the cases where  $x^2 \leq
1+n^2u$ and $x^2 > 1+n^2u$.

\subsection{First case: if $x^2 \leq 1+n^2u$}

From $x \geq 1 + 2^{p/2}u \geq 1 + nu$, we deduce
\[
x^n \geq (1 + nu)^n > 1+n^2u,
\]
so that, from Remark~\ref{rem:rounding-downwards}, we can assume that
\[
\widehat{x}_{n-1} \cdot{} x > (1+n^2u)
\]
(otherwise, at least one rounding was done downwards, which implies Theorem~\ref{thm:main-theorem}). Therefore
\begin{itemize}
   \item if $\widehat{x}_{n-1}x < 2$, then $\overline{\widehat{x}_{n-1}x}\geq (1+n^2u)$, so that, from Remark~\ref{rem:significand-large-enough}, $x^n \leq (1 + (n-1) \cdot u) \cdot x^n$;
   \item if $\widehat{x}_{n-1}x \geq 2$, then let $k$ be the smallest integer such that $\widehat{x}_{k-1}x \geq 2$. Notice that since we have assumed that $x^2 \leq 1+n^2u$, we necessarily have $k \geq 3$. We have
   \[
   \widehat{x}_{k-1} \geq \frac{2}{x} \geq \frac{2}{\sqrt{1+n^2u}},
   \]
   hence
  \begin{equation}
  \label{eq:boundonxk-2x}
   \widehat{x}_{k-2} \cdot x \geq \frac{2}{\sqrt{1+n^2u} \cdot (1+u)}.
   \end{equation}
   Now, define
   \[
   \alpha_p = \sqrt{\left(\frac{2^{p+1}}{2^p+1}\right)^{2/3}-1}.
   \]
   For all $p \geq 5$, $\alpha_p \geq \alpha_5 = 0.74509\cdots{}$, and $\alpha_p \leq \sqrt{2^{2/3}-1} = 0.7664209\cdots{}$.
   If
   \begin{equation}
   n \leq \alpha_p \cdot{} 2^{p/2},
   \end{equation}
   then
   \[
   1+n^2u \leq \left(\frac{2^{p+1}}{2^p+1}\right)^{2/3},
   \]
   so that
   \[
   (1+n^2u)^{3/2} \cdot (1+u) \leq 2,
   \]
   so that
   \[
   \frac{2}{\sqrt{1+n^2u} \cdot (1+u)} \geq 1+n^2 u.
   \]
   Therefore, from (\ref{eq:boundonxk-2x}), we have
   \[
    \widehat{x}_{k-2} \cdot x \geq 1+n^2 u.
    \]
    Also, $ \widehat{x}_{k-2} \cdot x$ is less than 2, since $k$ was assumed to be the smallest integer such that $\widehat{x}_{k-1}x \geq 2$. Therefore
     \[
    \overline{ \widehat{x}_{k-2} \cdot x } \geq 1+n^2 u.
    \]
   Which implies, by Remark~\ref{rem:significand-large-enough}, that $x^n \leq (1 + (n-1) \cdot u) \cdot x^n.$ So, to summarize this first case, if $x^2 \leq 1+n^2u$ and $n \leq \alpha_p \cdot{} 2^{p/2}$, then the conclusion of Theorem~\ref{thm:main-theorem} holds.
   \end{itemize}

\subsection{Second case: if $x^2 > 1+n^2u$}

First, if $x^2 < 2$ then we deduce from Remark~\ref{rem:significand-large-enough} that $x^n \leq (1 + (n-1) \cdot u) \cdot x^n$. The case $x^2 = 2$ is impossible ($x$ is a floating-point number, thus it cannot be irrational). Therefore let us now assume that $x^2 > 2$. We also assume that $x ^2 < 2 + 2n^2u$ (otherwise, we would have $\overline{ (x^2) }\geq 1 + n^2u$, so that we could apply Remark~\ref{rem:significand-large-enough}). Hence, we have
\[ \sqrt{2} < x < \sqrt{2+2n^2u}.\]
From this we deduce
\[
x^{n-1} < (2+2n^2u)^{\frac{n-1}{2}},
\]
therefore, using Theorem~\ref{thm:classical-theorem},
\[
\widehat{x}_{n-1} < (2+2n^2u)^{\frac{n-1}{2}} \cdot{} (1+u)^{n-2},
\]
which implies
\begin{equation}
\label{eq:bound-hat-xn-1xn}
x \cdot{} \widehat{x}_{n-1} < (2+2n^2u)^{n/2} \cdot{} (1+u)^{n-2}.
\end{equation}

Define
\[
\beta = \sqrt{2^{1/3}-1} = 0.5098245285339\cdots{}
\]
If $n \leq \beta \cdot 2^{p/2}$ then $2+2n^2u \leq 2^{4/3}$, so that we find
\begin{equation}
\label{eq:bound-on-the-bound-hat-xn}
(2+2n^2u)^{n/2} \cdot{} (1+u)^{n-2} \leq 2^{2n/3} \cdot (1+u)^{n-2}.
\end{equation}
\begin{itemize}
     \item if $n=3$, the bound on $x \cdot{} \widehat{x}_{n-1}$ derived from (\ref{eq:bound-hat-xn-1xn}) and~(\ref{eq:bound-on-the-bound-hat-xn}) is equal to $4 \cdot{} (1+u)$. Therefore either  $x \cdot{} \widehat{x}_{n-1} < 4$, or $x \cdot{} \widehat{x}_{n-1}$ will be rounded downwards when computing $\widehat{x}_n$ (in which case we already know from Remark~\ref{rem:rounding-downwards} that the conclusion of Theorem~\ref{thm:main-theorem} holds);
     \item if $n \geq 4$, consider function
     \[
     g(t) = 2^{t-1} - 2^{2t/3}\left(1 + \frac{1}{2^p}\right)^{t-2} = 2^{2t/3} \left[2^{t/3-1} - \left(1 + \frac{1}{2^p}\right)^{t-2}\right].
     \]
     It is a continuous function, and it goes to $+\infty$ as $t \to +\infty$. We have:
     \[
     g(t) = 0 \Leftrightarrow t = \frac{\log(2) + 2 \log \left(1 + \frac{1}{2^p}\right)}{\frac{1}{3}\log(2) - \log \left(1 + \frac{1}{2^p}\right)}.
     \]
     Hence, function $g$ has one root only, and as soon as $p \geq 5$, that root is strictly less than $4$. From this, we deduce that if $p \geq 5$, then $g(t) > 0$ for all $t \geq 4$. Hence, using (\ref{eq:bound-hat-xn-1xn}) and~(\ref{eq:bound-on-the-bound-hat-xn}), we deduce that if $p \geq 5$ then $x \cdot{} \widehat{x}_{n-1} < 2^{n-1}$.
\end{itemize}
Now that we have shown that\footnote{Unless $n=3$ and $x \cdot{} \widehat{x}_{n-1} \geq 4$ but in that case we have seen that the conclusion of Theorem~\ref{thm:main-theorem} holds.} if $n \leq \beta \cdot{} 2^{p/2}$ then
\[
x \cdot{} \widehat{x}_{n-1} < 2^{n-1},
\]
let us define $k$ as the smallest integer for which $x \cdot{} \widehat{x}_{k-1} < 2^{k-1}$. We now know that $k \leq n$, and (since we are assuming $x^2 > 2$), we have $k \geq 3$. The minimality of $k$ implies that $x \cdot{} \widehat{x}_{k-2} \geq 2^{k-2}$, which implies that $\widehat{x}_{k-1} = \RN(x \cdot{} \widehat{x}_{k-2})  \geq 2^{k-2}$. Therefore, $\widehat{x}_{k-1}$ and $x \cdot \widehat{x}_{k-1}$ belong to the same binade, therefore,
\begin{equation}
\label{eq:greater-than-sqrt-2}
\overline{x \cdot \widehat{x}_{k-1}} \geq x > \sqrt{2}.
\end{equation}
The constraint $n \leq \beta \cdot 2^{p/2}$ implies 
\begin{equation}
\label{less-than-cube-root-2}
1 + n^2 u \leq 1 + \beta^2 = 2^{1/3} < \sqrt{2}.
\end{equation}
By combining (\ref{eq:greater-than-sqrt-2}) and (\ref{less-than-cube-root-2}) we obtain
\[
\overline{x \cdot \widehat{x}_{k-1}} \geq 1 + n^2 u.
\]
Therefore, using Remark~\ref{rem:significand-large-enough}, we deduce that $\widehat{x}_n \leq (1 + (n-1) \cdot u) \cdot x^n$.

\subsection{Combining both cases}

One easily sees that for all $p \geq 5$, $\alpha_p$ is larger than $\beta$. Therefore, combining the conditions found in the cases $x^2 \leq 1 + n^2u$ and $x^2 > 1 + n^2u$, we deduce that if $p \geq 5$ and $n \leq \beta \cdot{} 2^{p/2}$, then for all $x$,
\[
(1 - (n-1) \cdot u) \cdot x^n \leq \widehat{x}_n \leq (1 + (n-1) \cdot u) \cdot x^n.
\]
Q.E.D.

Notice that the condition $n \leq \beta \cdot{} 2^{p/2}$ is not a huge constraint. The table below gives the maximum value of $n$ that satisfies that condition, for the various binary formats of the IEEE 754-2008 Standard for Floating-Point Arithmetic.

\[
\begin{array}{|r|l|}
\hline
p & n_\mathrm{max} \\
\hline\hline
24 & 2088 \\
\hline
53 & 48385542 \\
\hline
113 & 51953580258461959 \\
\hline
\end{array}
\]

For instance, in the binary32/single precision format, with the smallest $n$ larger than that maximum value (i.e., $2089$), $x^n$ will underflow as soon as $x \leq 0.95905406$ and overflow as soon as $x \geq 1.0433863$. In the binary64/double precision format, with $n = 4385543$, $x^n$ will underflow as soon as $x \leq 0.999985359$ and overflow as soon as $x \geq 1.000014669422$. With the binary113/quad precision format, the interval in which function $x^n$ does not under- or overflow is even narrower and, anyway, computing $x^{51953580258461959}$ by Algorithm~\ref{alg:naive-power} would at best require months of computation on current machines.

\section{Is the bound of Theorem~\ref{thm:main-theorem} tight?} \label{sec:thm}

For very small values of $p$, it is possible to check all possible
values of $x$ (we can assume $1 \leq x < 2$, so that we need to check
$2^{p-1}$ different values), using a Maple program that simulates a
precision-$p$ floating-point arithmetic. Hence, for small values of
$p$ and reasonable values of $n$ it is possible to compute the actual
maximum relative error of Algorithm~\ref{alg:naive-power}. For
instance, Tables~\ref{tab:p=8} and~\ref{tab:p=9} present the actual
maximum relative errors for $p = 8$ and $9$, respectively, and various
values of $n$. 

\begin{table}[!ht]
\caption{Actual maximum relative error of
  Algorithm~\ref{alg:naive-power} assuming precision $p = 8$, compared
  with the usual bound $\gamma_{n-1}$ and our bound $(n-1)u$. The term
  $n_{max}$ designs the largest value of $n$ for which
  Theorem~\ref{thm:main-theorem} holds, namely $\sqrt{2^{1/2}-1} \cdot
  2^{p/2}$} 
\label{tab:p=8}
\[
\begin{array}{|l|c|c|c|}
\hline
n & \mathrm{actual~maximum} & \gamma_{n-1} & \mathrm{our~bound} \\
\hline\hline
3 & 1.35988 u & 2.0157 u & 2u \\
\hline
4 & 1.73903 u & 3.0355 u & 3u \\
\hline
5 & 2.21152 u & 4.06349 u & 4u \\
\hline
6 & 2.53023 u & 5.099601 u & 5u \\
\hline
7 & 2.69634 u & 6.1440 u & 6u \\
\hline
8 = n_{max} & 3.42929 u & 7.1967 u & 7u \\
\hline
\end{array}
\]
\end{table}

\begin{table}[!ht]
\caption{Actual maximum relative error of
  Algorithm~\ref{alg:naive-power} assuming precision $p = 9$, compared
  with the usual bound $\gamma_{n-1}$ and our bound $(n-1)u$. The term
  $n_{max}$ designs the largest value of $n$ for which
  Theorem~\ref{thm:main-theorem} holds, namely $\sqrt{2^{1/2}-1} \cdot
  2^{p/2}$} 
\label{tab:p=9}
\[
\begin{array}{|l|c|c|c|}
\hline
n & \mathrm{actual~maximum} & \gamma_{n-1} & \mathrm{our~bound} \\
\hline\hline
6 & 2.677 u & 5.049 u & 5u \\
\hline
7 & 2.975 u & 6.071 u & 6u \\
\hline
8  & 3.435 u & 7.097 u & 7u \\
\hline
9 & 4.060 u & 8.1269 u & 8u \\
\hline
10 & 3.421 u & 9.1610 u & 9u \\
\hline
11 = n_{max} & 3.577 u &  10.199 u & 10u \\
\hline
\end{array}
\]
\end{table}

For larger values, we have some results (notice that beyond single precision---$p=24$---exhaustive testing is out of reach):

\begin{itemize}
  \item for single precision arithmetic ($p=24$) and $n=6$, the actual largest relative error is $4.328005619 u$. It is attained for $x = 8473808/2^{23} \approx 1.010156631$; 
     \item for double precision arithmetic ($p=53$) and $n=6$, although finding the actual largest relative error is out of reach, we could find an interesting case: for $x = 4507062722867963/2^{52} \approx 1.0007689616715527147761$, the relative error is $4.7805779\cdots{} u$
  \item for quad precision arithmetic ($p=113$) and $n=6$, although finding the actual largest relative error is out of reach, we could find an interesting case: for \[\begin{array}{c}
  x = 5192324351407105984705482084151108/2^{112} \\ \approx 1.0000052949345978099886352037496365983, \end{array} \] the relative error is $4.8827888\cdots{} u$
   \item for single precision arithmetic ($p=24$) and $n=10$, the actual largest relative error is $7.059603149 u$. It is attained for $x = 8429278/2^{23} \approx 1.004848242$;
   \item for double precision arithmetic ($p=53$) and $n=10$, although finding the actual largest relative error is out of reach, we could find an interesting case: for $x = 4503796447992526/2^{52} \approx 1.00004370295725975026$, the relative error is $7.9534189\cdots{} u$.
\end{itemize}

Notice that we can use the maximum relative error of single precision and ``inject it'' in the inductive reasoning that led to Theorem~\ref{thm:classical-theorem} to show that \emph{in single-precision arithmetic, and if $n \geq 10$} then
\[
(1-7.06u)(1-u)^{n-10} x^n \leq \widehat{x}_n \leq (1+7.06u)(1+u)^{n-10}x^n.
\]
Then, by replacing $u$ by $2^{-24}$ and through an elementary study of the function
\[
\varphi(t) = \left[(1+7.06 \cdot 2^{-24})(1+2^{-24})^{t-10}-1\right] \cdot 2^{24}-t
\]
one easily deduces that for $10 \leq n \leq 2088$, we always have
\[
\left| \frac{\widehat{x}_n - x^n}{x^n} \right| \leq (n-2.8104) \cdot u.
\]

\section{What about iterated products ?} \label{sec:prod}
\label{sec:iterated-products}
Assume now that, still in precision-$p$ binary FP arithmetic, we wish to evaluate the product
\[
a_1 \cdot{} a_2 \cdots{} \cdots{} \cdot a_n,
\]
of $n$ floating-point numbers. We assume that the product is evaluated as
\[
\RN( \cdots{} \RN(\RN(a_1 \cdot{} a_2) \cdot{} a_3) \cdot{} \cdots{} )
\cdot{} a_n) .
\]

Define $\pi_k$ as the exact value of $a_1 \cdots a_k$, and $\widehat{\pi}_k$ as the computed value. As already discussed in Section~\ref{sec:relative-error-classical-general}, ve have
\begin{equation}
(1 - u)^{n-1} \pi_n \leq \widehat{\pi}_n \leq (1+u)^{n-1} \pi_n,
\end{equation}
which implies that the relative error $|\pi_n - \widehat{\pi}_n|/\pi_n$ is upper-bounded by $\gamma_{n-1}$. We conjecture that the error is upper-bounded by $(n-1) u$. Let us now show how to build $a_1$, $a_2$, \ldots{}, $a_n$ so that the relative error becomes extremely close to $(n-1) \cdot u$. 

Define $a_1 = 1 + k_1 \cdot 2^{-p+1}$, and $a_2 = 1 + k_2 \cdot 2^{-p+1}$. We have
\[
\pi_2 = a_1 a_2 = 1 + (k_1+k_2) \cdot 2^{-p+1} + k_1k_2 \cdot 2^{-2p+2}.
\]
If $k_1$ and $k_2$ are not too large, $ 1 + (k_1+k_2) \cdot 2^{-p+1} $ is a FP number. To maximize the relative error, we wish $k_1+k_2$ to be as small as possible, while $k_1k_2 \cdot 2^{-2p+2}$ is as close as possible to $2^{-p}$. Hence a natural choice is
\[
k_1 = k_2 = \left\lfloor 2^{\frac{p}{2}-1} \right\rfloor,
\]
which gives $\widehat{\pi}_2 < \pi_2$.
Now, if at step $i-1$ we have
\[
\widehat{\pi}_i = 1 + g_i \cdot 2^{-p+1}, \mathrm{~with~} \widehat{\pi}_i < \pi_i,
\]
we choose $a_{i+1}$ of the form $1 + k_{i+1}2^{-p+1}$, with
\begin{itemize}
    \item $k_{i+1} = \left\lceil \frac{2^{p-2}}{g_i} - 1 \right\rceil$ if $g_i \leq 2^{\frac{p}{2}-1}$;
    \item $k_{i+1} = -\left\lfloor \frac{2^{p-2}}{g_i} + 1 \right\rfloor$ otherwise.
\end{itemize}

For instance, in single precision ($p=24$), the first values $a_i$ generated by this strategy are
\[
\begin{array}{lll}
a_1 &=& 4097/4096 \\
a_2 &=& 4097/4096 \\
a_3 &=& 8387583/8388608 \\
a_4 &=& 8387241/8388608 \\
a_5 &=& 262221/262144 \\
a_6 &=& 8387601/8388608 \\
a_7 &=& 8387279/8388608 \\
\end{array}
\]

Table~\ref{tab:built-iterated-products} gives examples of the relative
errors achieved with the values $a_i$ generated by this method, for
various values of $p$ and $n$. As one can easily see, the relative
error is always very close to, but less than $(n-1) \cdot u$.

\begin{table}[htb]
\caption{Relative errors achieved with the values $a_i$ generated by our method of Section~\ref{sec:iterated-products}.}
\label{tab:built-iterated-products}
\[
\begin{array}{|r|r|l|}
\hline
p & n & \mathrm{relative~error} \\
\hline\hline
24 & 10 & 8.99336984\cdots{}  u \\
\hline
24 & 100 &  98.9371972591\cdots{}  u \\
\hline\hline
53 & 10 & 8.99999972447 \cdots{} u \\
\hline
53 & 100 & 98.9999970091 \cdots{} u \\
\hline\hline
113 & 10 & 8.99999999999999973119 \cdots{} u \\
\hline
113 & 100 & 98.99999999999999701662 \cdots{} u \\
\hline
\end{array}
\]
\end{table}

\section{Conclusion}

We have shown that, under mild conditions, the relative error of the computation of $x^n$ in floating-point arithmetic using the ``naive'' algorithm is upper bounded by $(n-1) \cdot{} u$. This bound is simpler and slightly better than the previous bound. We conjecture that the same bound holds in the more general case of the computation of the product of $n$ floating-point numbers. In that case, we have provided examples that show that the actual error can be very close to $(n-1) \cdot u$.

%
%
%
\bibliographystyle{plain}


\end{document}